\documentclass{article}
\usepackage{amsmath}
\usepackage{amsthm}
\usepackage{amssymb}
\usepackage[dvips]{graphicx}
\usepackage{latexsym}
\usepackage{mathtools}
\usepackage{geometry}
 \renewcommand{\equation}

\newtheorem{prop}{Proposition}[section]
\newtheorem{lem}{Lemma}[section]
\newtheorem{thm}{Theorem}[section]

\newtheorem{fact}{Fact}[section]
\newtheorem{corresp*}{Correspondence}[section]
\newtheorem{facts*}{Facts}[section]

\theoremstyle{definition}
\newtheorem{rmk}{Remark}[section]
\newtheorem{defini}{Definition}[section]

\title{Generalized torsions in once punctured torus bundles}
\author{Nozomu Sekino}
\date{}

\begin{document}
\maketitle

\begin{abstract}
A generalized torsion in a group, an non-trivial element such that some products of its conjugates is the identity. 
This is an obstruction for a group being bi-orderable. 
Though it is known that there is a non bi-orderable group without generalized torsions, it is conjectured that 3-manifold groups without generalized torsions are bi-orderable. 
In this paper, we find generalized torsions in the fundamental groups of once punctured torus bundles which are not bi-orderable. 
Our result contains a generalized torsion in a tunnel number two hyperbolic once punctured torus bundle.
\end{abstract}

\section{Introduction} \label{sec1}
A group $\Gamma$ is called a {\it bi-orderable} group if it admits total order $<$ which is invariant under multiplication to both sides, i.e. if $g<g'$ then $fgh<fg'h$ for every $f,h\in \Gamma$. 
As a convention, the trivial group is regarded as a bi-orderable group. 
There is one obstruction for groups being bi-orderable, called {\it generalized torsions}. 
An non-trivial element $g\in \Gamma$ is called generalized torsion if some product of conjugates of $g$ is the identity of $\Gamma$. 
In general, there exist groups which are not bi-orderable and have no generalized torsions \cite{bludov1}, \cite{bludov2}, \cite{mura}. 
However, it is conjectured that bi-orderability and having no generalized torsions is equivalent for the class of 3-manifold groups \cite{motegi1}. 
We have many affirmative answers for this conjecture, for some geometric manifolds, for some manifolds with non-trivial geometric decompositions, for many link complements and their Dehn fillings \cite{ito1},\cite{ito2},\cite{motegi1},\cite{motegi2},\cite{naylor},\cite{teragaito1},\cite{teragaito2}. 

For the complements of fibered links (not necessarily in $S^3$), the following important sufficient condition and necessary condition (not necessary and sufficient condition) for the bi-orderability are known.:

\begin{thm} \label{realpositive} (A consequence of Theorem 2.7 of \cite{perron})\\
Suppose that the isomorphism on the first homology group of the fiber surface induced by the monodromy of a fibered knot has only real positive eigenvalues. 
Then the fundamental group of the complement of this fibered knot is bi-orderable. 
\end{thm}

\begin{thm} \label{one} (A consequence of Theorem 1.4 of \cite{clay})\\
Suppose that the fundamental group of the complement of a fibered knot is bi-orderable. 
Then the isomorphism on the first homology group of the fiber surface induced by the monodromy of the fibered knot has at least one real positive eigenvalue. 
\end{thm}

In this paper, we consider {\it once punctured torus bundles}, surface-bundles over circles each of whose fiber is torus with connected boundary. 
In this case, the conditions above are necessary and sufficient condition for the bi-orderability. 
We will show that if the fundamental group of a once punctured torus bundle is not bi-orderable then it admits a generalized torsion. 
This adds the class of once punctured torus bundles to the list of 3-manifolds where the conjecture holds. 

The rest of this paper is organized as follows. 
In Section~\ref{sec2}, we represent the fundamental groups as oriented based paths and the conjugation classes of the fundamental groups as  oriented free loops. 
We review the relations between theses and some operations. 
In Section~\ref{sec3}, we review that generalized torsions in the fundamental group of fiber-bundles over circles are in that of the fibers and satisfy some relation in terms of monodromies. 
In Section~\ref{sec4}, we assign words and cyclic words to oriented based paths and oriented free loops in once punctured torus, respectively. 
After these preparations, we prove the main result in Section~\ref{sec5}. 
Thanks to Theorem~\ref{realpositive} and Theorem~\ref{one}, we know when the fundamental group of a once punctured torus bundle is not bi-ordearble. 
We will find a generalized torsion in this case. 
At last, as a remark, we show that our result contains a generalized torsion in the fundamental group of a tunnel number two hyperbolic once punctured torus bundle.

In this paper, we do not distinguish the homotopy classes of paths or loops with their representatives. The words in groups are read from left to right. For a group $\Gamma$, $C(\Gamma)$ denotes the set of conjugacy classes of $\Gamma$.

\section*{Acknowledgements}
The author would like to thank professor Motegi for introducing the area of generalized torsions to him, and giving him many helpful comments. 

\section{Based paths and free loops} \label{sec2}
For a topological space $X$ and a base point $*\in X$, the fundamental group $\pi_{1}(X,*)$ is defined as the set of the homotopy types of oriented based paths $([0,1],\{0,1\})\longrightarrow (X,*)$ with multiplications as concatenations. 
Let $F: \pi_{1}(X,*)\longrightarrow C\left(\pi_{1}(X,*)\right)$ be a map which sends an element to its conjugacy class. 
Topologically, $C\left(\pi_{1}(X,*)\right)$ is regarded as the set of the homotopy types of oriented free loops $S^{1}\longrightarrow X$ and $F$ as the operation which identifies the endpoints of paths and forgets the identified special point. 

\begin{defini}
For an oriented free loop $l$ and a point $p$ on $l$, $l\{p\}$ denotes the oriented path starting at $p$, following $l$ and ending at $p$. 
Moreover, for another point $q$ on $l$, which is not $p$, $l\{p,q\}$ denotes the subpath of $l$ whose initial endpoint is $p$ and terminal point is $q$. 
Note that $l\{p,q\}$ and $l\{q,p\}$ are different. 
\end{defini}

\begin{defini}
For an oriented path or an oriented free loop $I$, $\bar{I}$ denotes the oriented path or the oriented free loop obtained by reversing the orientation of $I$.
\end{defini}

\begin{defini}
For two oriented paths $I_{1}$ and $I_{2}$ such that the terminal point of $I_{1}$ and the initial point of $I_{2}$ are identical, $I_{1}-I_{2}$ denotes the concatenation of $I_{1}$ and $I_{2}$.
\end{defini}

\begin{defini}
Let $l_1, l_2 \in C\left(\pi_{1}(X,*)\right)$ be two free loops, and $\alpha$ an oriented path in $X$ whose initial endpoint $\alpha(0)$ is on $l_1$ and terminal endpoint $\alpha(1)$ is on $l_2$. 
Then $l\left(l_1,\alpha,l_2\right)$ denotes a new free loop obtained by forgetting the endpoints of $l_{1}\{\alpha(0)\}-\alpha-l_{2}\{\alpha(1)\}-\bar{\alpha}$. 
We call $l\left(l_1,\alpha,l_2\right)$ {\it the loop obtained by connecting} $l_1$ {\it and} $l_2$ {\it using} $\alpha$.
\end{defini}

\begin{lem} \label{connection}
Let $l_1, l_2 \in C\left(\pi_{1}(X,*)\right)$ be two free loops, and $\alpha$ an oriented path in $X$ whose initial endpoint $\alpha(0)$ is on $l_1$ and terminal endpoint $\alpha(1)$ is on $l_2$. 
Let $I_{1},I_{2} \in \pi_{1}(X,*)$ be elements satisfying $F(I_{i})=l_{i}$ for $i=1,2$. 
Then there exists an element $g\in \pi_{1}(X,*)$ such that $F(I_{1}g I_{2}g^{-1})=l\left(l_1,\alpha,l_2\right)$.
\end{lem}
\begin{proof}
Note that there exist oriented paths $\beta_{i}$ starting from $*$ and ending on $l_{i}$, denoted by $\beta_{i}(1)$, such that $\beta_{i} - l_{i}\{\beta_{i}(1)\}-\bar{\beta_{i}} =I_{i}$ for $i=1,2$ (under homotopies). 
Let $g\in \pi_{1}(X,*)$ be $\bar{\beta_{1}}-\bar{l_{1}}\{\beta_{1}(1), \alpha(0)\}-\alpha-\overline{l_{2}\{\beta_{2}(1),\alpha(1)\}}-\bar{\beta_{2}}$. 
Then we have $I_{1}g I_{2}g^{-1}=\beta_{1}-l\left(l_1,\alpha,l_2\right)\{\beta_{1}(1)\}-\bar{\beta_{1}}$ and $F(I_{1}g I_{2}g^{-1})=l\left(l_1,\alpha,l_2\right)$. See Figure~\ref{conjugation}.

\begin{figure}[htbp]
 \begin{center}
  \includegraphics[width=60mm]{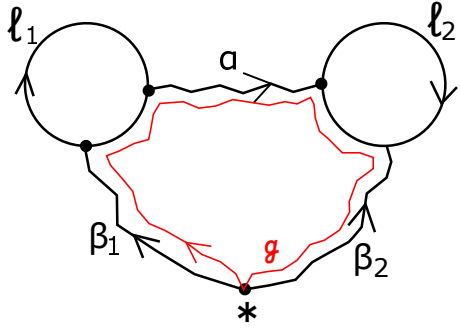}
 \end{center}
 \caption{A schematic of the based path $g$.}
 \label{conjugation}
\end{figure}

\end{proof}

\section{Generalized torsions in a fibered space} \label{sec3}
In this section, we show that generalized torsions in the fundamental group of a fiber bundle over a circle are same as the elements in the fundamental group of the fiber space satisfying some relations (Proposition~\ref{gentorinfib}).

\begin{defini}(Generalized torsion)\\
Let $\Gamma$ be a group. A non-trivial element $g$ is called a {\it generalized torsion} if there exist positive integer $n$ and elements $w_{1},\dots, w_{n}\in \Gamma$ such that $(w_{1}g{w_{1}}^{-1})\cdots (w_{n}g{w_{n}^{-1}})=1_{\Gamma}$, the identity of $\Gamma$.
\end{defini}

Let $(X,*)$ be a based topological space and $\phi$ a self-homeomorphism of $X$ fixing $*$. 
Let $(Y,*_{Y})$ be a topological space obtained from $X\times [0,1]$ by identifying $(x, 1)$ and $(\phi(x),0)$ for $x\in X$, and $*_{Y}$ comes from $(*,0)$. 
In this situation, $(Y,*_Y)$ is called the {\it $(X,*)$-bundle over a circle using $\phi$} and $\phi$ is called {\it monodromy}. 
Let $\tau \in \pi_{1}(Y,*_{Y})$ be the based path coming from $[0,1]\longrightarrow X\times [0,1]; t\longmapsto (*,t)$ and let $T$ be a subgroup of $\pi_{1}(Y,*_{Y})$ generated by $\tau$. 
Then it is known that $T$ is isomorphic to $\mathbb{Z}$, and that $\pi_{1}(X,*)$ injects in $\pi_{1}(Y,*_{Y})$. 
Under these notation, $\pi_{1}(Y,*_{Y})$ is isomorphic to $\pi_{1}(X,*)\rtimes T$, where the multiplication is defined as $(g_1,\tau ^{n_1})\cdot (g_2, \tau ^{n_2})=(g_{1}\tilde{\phi}^{n_1}(g_{2}), \tau ^{n_{1}+n_{2}})$, where $\tilde{\phi}$ is the induced map between $\pi_{1}(X,*)$ by $\phi$. 
Abusing notations, $(g,1_{T})$ for $g\in \pi_{1}(X,*)$ and $(1_{\pi_{1}(X,*)},\tau^{n})$ are written as $g$ and $\tau^{n}$, respectively. 
Note that for $g\in \pi_{1}(X,*)$, we have $\tau g \tau^{-1}=\tilde{\phi}(g)$ and $\tau^{-1} g \tau=\tilde{\phi}^{-1}(g)$ in $\pi_{1}(Y,*_{Y})$.\\

Suppose that $g\in \pi_{1}(Y,*_{Y})$ is a generalized torsion. 
Take positive integer $n$ and $w'_{1},\dots,w'_{n}\in \pi_{1}(Y,*_{Y})$ such that $(w'_{1}g{w'_{1}}^{-1})\cdots (w'_{n}g{w'_{n}}^{-1})=1_{\pi_{1}(Y,*_{Y})}$. 
By considering the projection $\pi_{1}(Y,*_{Y})\longrightarrow T$, we have $g\in \pi_{1}(X,*)$. 
Represent $w'_{i}$ as $g_{i_{1}}\tau^{n_{i_{1}}}\cdots g_{i_{k}}\tau^{n_{i_{k}}}$, where $g_{i_{j}}\in \pi_{1}(X,*)$ and $n_{i_{j}}\in \mathbb{Z}$ for $j=1,\dots k$. 
Then $(w'_{i}g{w'_i}^{-1})= \left( \prod^{k}_{j=1}\tilde{\phi}^{\left(\sum^{j}_{l=1} n_{i_l} \right)}(g_{i_j})\right) \tilde{\phi}^{\left( \sum^{k}_{l=1}n_{i_l} \right)}(g)  \left( \prod^{k}_{j=1}\tilde{\phi}^{\left(\sum^{j}_{l=1} n_{i_l} \right)}(g_{i_j})\right)^{-1}$ in $\pi_{1}(Y,*_{Y})$. 
Set $N_{i}$ to be $\sum^{k}_{l=1}n_{i_l}$ and $w_{i}\in \pi_{1}(X,*)$ to be $\left( \prod^{k}_{j=1}\tilde{\phi}^{\left(\sum^{j}_{l=1} n_{i_l} \right)}(g_{i_j})\right)$. 
Then we have an equation $\left(w_{1}\tilde{\phi}^{N_1}(g){w_{1}}^{-1}\right)\cdots \left(w_{n}\tilde{\phi}^{N_n}(g){w_{n}}^{-1}\right)=1_{\pi_{1}(Y,*_{Y})}$ in $\pi_{1}(Y,*_Y)$. 
Since the left-hand side is a composition of elements in $\pi_{1}(X,*)$ and $\pi_{1}(X,*)$ injects in $\pi_{1}(Y,*_Y)$, we have an equation $\left(w_{1}\tilde{\phi}^{N_1}(g){w_{1}}^{-1}\right)\cdots \left(w_{n}\tilde{\phi}^{N_n}(g){w_{n}}^{-1}\right)=1_{\pi_{1}(X,*)}$ in $\pi_{1}(X,*)$. 

Conversely, suppose that there exist $g\in \pi_{1}(X,*)$, positive integer $n$, integers $N_{1},\dots,N_{n}$ and elements $w_{1}\dots ,w_{n}\in \pi_{1}(X,*)$ such that $\left(w_{1}\tilde{\phi}^{N_1}(g){w_{1}}^{-1}\right)\cdots \left(w_{n}\tilde{\phi}^{N_n}(g){w_{n}}^{-1}\right)=1_{\pi_{1}(X,*)}$ holds in $\pi_{1}(X,*)$. 
Then $\left(w_{1}\tau^{N_1}g\tau^{-N_{1}}{w_{1}}^{-1}\right)\cdots \left(w_{n}\tau^{N_n}g\tau^{-N_n}{w_{n}}^{-1}\right)=1_{\pi_{1}(Y,*_{Y})}$ holds in $\pi_{1}(Y,*_{Y})$. 
This implies that $g$ is a generalized torsion.\\\\

We state the above argument as a Proposition:
\begin{prop} \label{gentorinfib}
Let $(Y,*_{Y})$ be the $(X,*)$-bundle over a circle using $\phi$ as the monodromy. 
An element $g\in \pi_{1}(Y,*_Y)$ is a generalized torsion if and only if $g\in \pi_{1}(X,*)$ and there exist positive integer $n$, integers $N_{1},\dots , N_{n}$ and elements $w_{1},\dots,w_{n}\in \pi_{1}(X,*)$ such that $\left(w_{1}\tilde{\phi}^{N_1}(g){w_{1}}^{-1}\right)\cdots \left(w_{n}\tilde{\phi}^{N_n}(g){w_{n}}^{-1}\right)=1_{\pi_{1}(X,*)}$ holds in $\pi_{1}(X,*)$. 
\end{prop}

\section{Paths and loops in a once punctured torus} \label{sec4}
Let $(\Sigma, *)$ be a torus with connected boundary with a base point on the boundary. 
In this section, we review the way to represent oriented based paths, elements of $\pi_{1}(\Sigma,*)$, and oriented free loops, elements of $C\left( \pi_{1}(\Sigma,*) \right)$. 
Take two disjoint properly embedded oriented arcs $I_x$ and $I_y$ such that they are disjoint from $*$ and cut $\Sigma$ into a disk as in Figure~\ref{arcsystem}. 

\begin{figure}[htbp]
 \begin{center}
  \includegraphics[width=30mm]{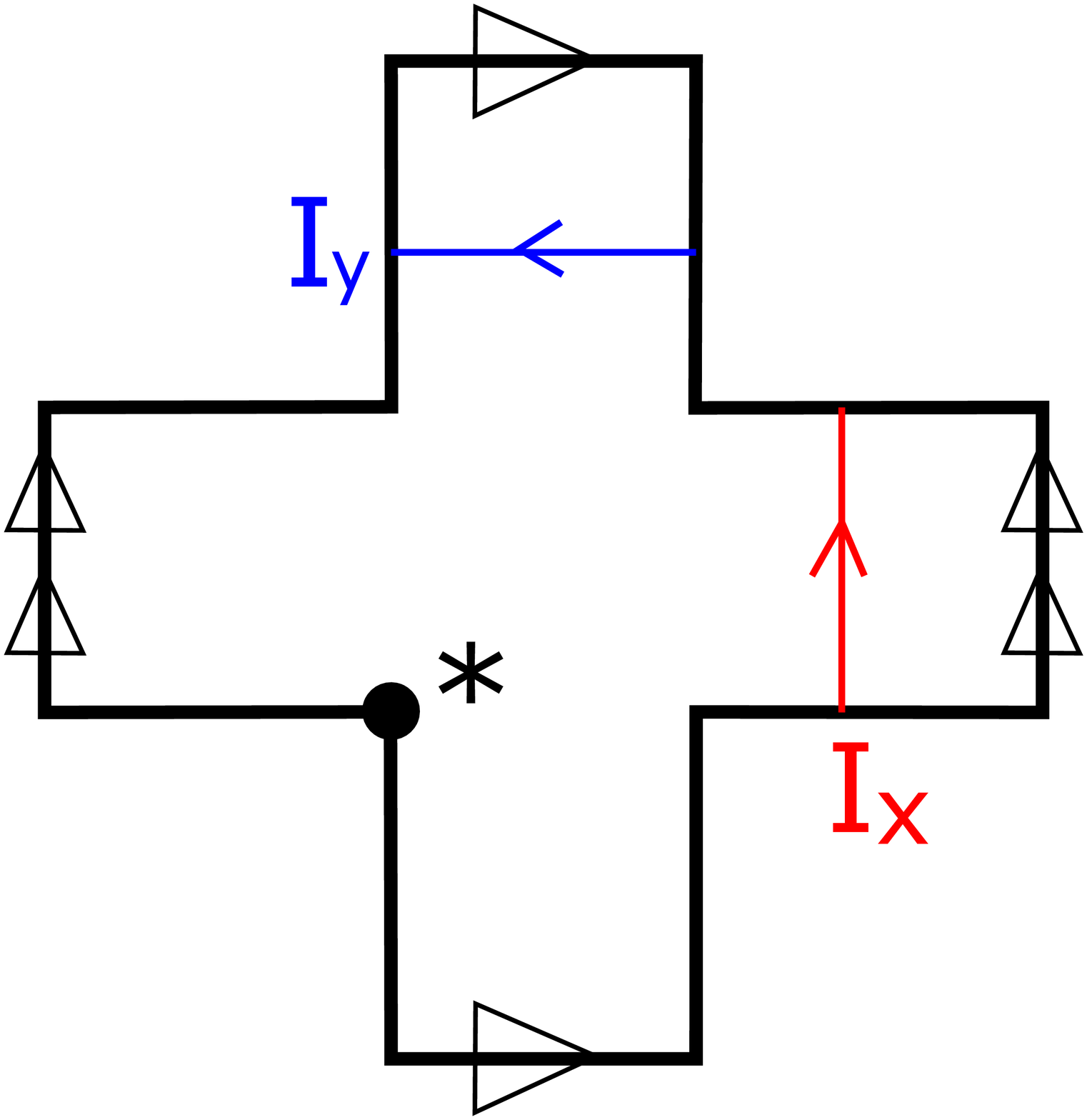}
 \end{center}
 \caption{Oriented arcs $I_x$, $I_y$ on $\Sigma$.}
 \label{arcsystem}
\end{figure}

\subsection{Based paths on $\Sigma$} \label{subsec1}
For an oriented based path $I\in \pi_{1}(\Sigma, *)$, we give a word $w(I)$, which is an element of the free group $G$ of rank $2$ generated by the alphabets $\{x,y\}$, as follows:
Take the empty word at first, and follow $I$. 
If $I$ hits $I_x$ from the left (or right) of $I_x$, then we add $x$ (or $x^{-1}$, respectively) to the right of the word we have, 
and if $I$ hits $I_y$ from the left (or right) of $I_y$, then we add $y$ (or $y^{-1}$, respectively) to the right of the word we have. 
Repeat this operation until we reach the terminal point of $I$. 
Note that this $w(\cdot)$ gives the isomorphism between $\pi_{1}(\Sigma, *)$ and $G$. 

\subsection{Free loops on $\Sigma$} \label{subsec2}

For an oriented free loop $l\in C\left( \pi_{1}(\Sigma, *)\right)$, we give a cyclic word $W(l)$, which is an element of $C(G')$ the set of the conjugacy classes of the free group $G'$ of rank $2$ generated by the alphabets $\{X,Y\}$, as follows: 
Take the empty word at first, and take a point $*_l$ on $l$ and follow $l\{*_{l}\}$. 
If $l\{*_{l}\}$ hits $I_x$ from the left (or right) of $I_x$, then we add $X$ (or $X^{-1}$, respectively) to the right of the word we have, 
and if $l\{*_{l}\}$ hits $I_y$ from the left (or right) of $I_y$, then we add $Y$ (or $Y^{-1}$, respectively) to the right of the word we have. 
Repeat this operation until we reach the terminal point of $I$. 
Then connect the last of the word to the first of the word so that we get a cyclic word. 
Note that this $W(\cdot)$ gives the isomorphism as sets between $C\left( \pi_{1}(\Sigma, *) \right)$ and $C(G')$. 
Note also that for an oriented based path $I\in \pi_{1}(\Sigma,*)$, $W\left( F(I) \right)$ is obtained from $w(I)$ by changing $x^{\pm1}$ and $y^{\pm1}$ to $X^{\pm1}$ and $Y^{\pm1}$ respectively, and connecting the last and the first of the word. 
Due to space limitation, we represent cyclic words in lines, and they should be regarded as cyclic. 

We give oriented simple loops on $\Sigma$ names as follows, where ``simple'' means having no self-intersections. 
Take two oriented simple loops $\alpha$ and $\beta$ as in Figure~\ref{alphabeta}. 

\begin{figure}[htbp]
 \begin{center}
  \includegraphics[width=35mm]{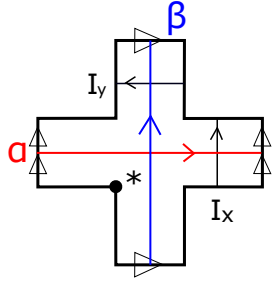}
 \end{center}
 \caption{Loops $\alpha$ and $\beta$.}
 \label{alphabeta}
\end{figure}

\begin{defini}
Let $p$ and $q$ are relative prime integers. 
Write $|p|$ parallel (disjoint) $\alpha$ loops and give them the same orientation as $\alpha$ if $p\geq 0$, and the opposite orientations if $p<0$. 
And write $|q|$ parallel (disjoint) $\beta$ loops and give them the same orientation as $\beta$ if $q\geq 0$, and the opposite orientations if $q<0$. 
Then resolve the intersection points so that the orientations are compatible. 
Since $p$ and $q$ are relative prime, we get an oriented simple loop. 
We call this loop $L(p,q)$. See Figure~\ref{loop_example} for example. 
It is well-known that every oriented simple free loop on $\Sigma$ is homotopic to $L(p,q)$ for the unique relative prime integers $p$ and $q$.
\end{defini}

\begin{figure}[htbp]
 \begin{center}
  \includegraphics[width=100mm]{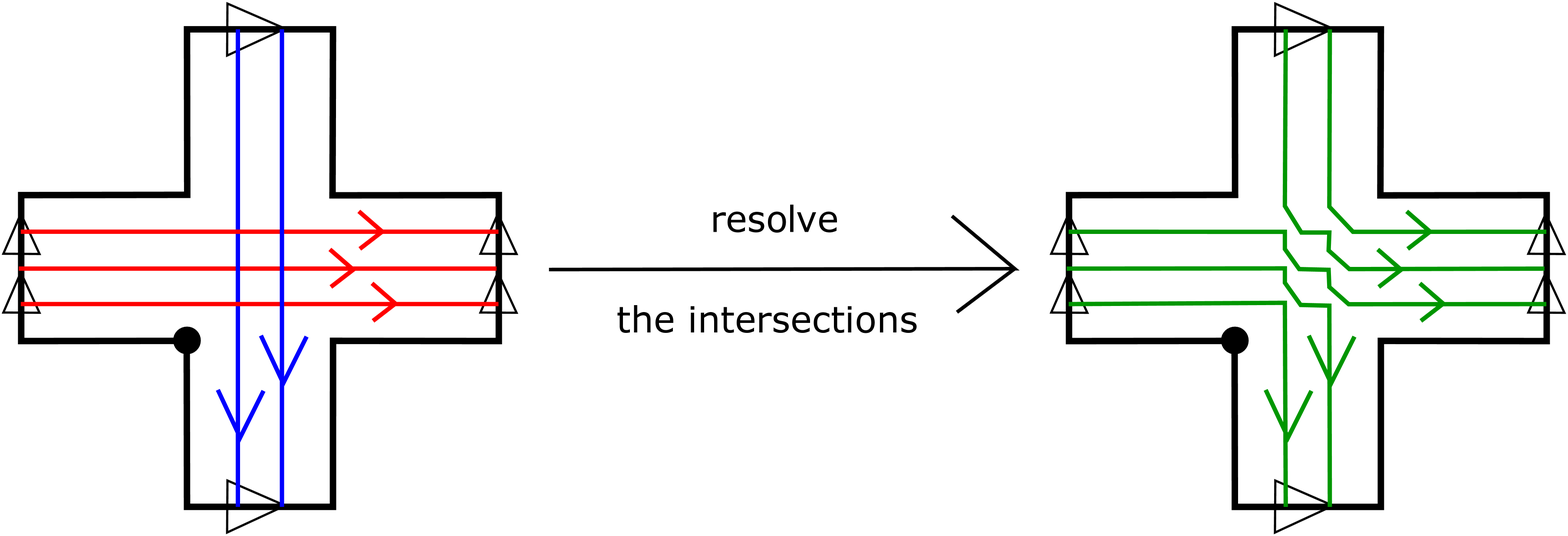}
 \end{center}
 \caption{L(3,-2)}
 \label{loop_example}
\end{figure}

\section{Once punctured torus bundles} \label{sec5}
Let $(\Sigma, *)$ be a torus with connected boundary with a base point on the boundary, and take the same arcs $I_x$ and $I_y$ as in Figure~\ref{arcsystem} and the same loops $\alpha$ and $\beta$ as in Figure~\ref{alphabeta}.  
Let $(M_{\phi},*_{M_{\phi}})$ be the $\Sigma$-bundle over a circle with orientation preserving and boundary fixing monodromy $\phi$. 
The 3-manifolds admitting such structures are called {\it once punctured torus bundles}, (OPTB for short). 
The representation matrix under the basis $\{ [\alpha], [\beta]\}$, which are the homology classes of $\alpha$ and $\beta$, of the induced isomorphism of $H_{1}(\Sigma;\mathbb{Z})$ by $\phi$ is denoted by $A_{\phi} \in SL_{2}(\mathbb{Z})$. 
Then the following Fact are known:
\begin{fact} \label{monodromy} (Lemma 3 of \cite{morimoto})\\
Two $\Sigma$-bundles over a circle $(M_{\phi},*_{M_{\phi}})$ and $(M_{\psi},*_{M_{\psi}})$ are homeomorphic (possibly reversing the orientation) if and only if $A_{\phi}$ is conjugate to $A_{\psi}$ or $A^{-1}_{\psi}$ in $GL_{2}(\mathbb{Z})$.
\end{fact}

Since the eigenvalues of $A\in SL_{2}(\mathbb{Z})$ is the roots of $x^{2}-(trA)x+1=0$, Theorem~\ref{realpositive} implies that $\pi_{1}(M_{\phi},*_{M_{\phi}})$ is bi-orderable (and thus admits no generalized torsions) if $trA_{\phi}\geq 2$. 
Moreover, Theorem~\ref{one} implies that $\pi_{1}(M_{\phi},*_{M_{\phi}})$ is not bi-orderable if $trA_{\phi}<2$. 
We want to show that $\pi_{1}(M_{\phi},*_{M_{\phi}})$ admits a generalized torsion if $trA_{\phi}<2$. 

Under Fact~\ref{monodromy}, we assume that representing matrix $A_{\phi}$ is of the form of the following:

\begin{fact}(Lemma 4 of \cite{morimoto})\\
Every element of $SL_{2}(\mathbb{Z})$ is conjugate in $GL_{2}(\mathbb{Z})$ to the matrix $A$ such that $A_{1,1}\cdot A_{2,2}\geq 0$ and $|A_{1,1}|\geq |A_{2,2}|$, where $A_{i,j}$ is the $(i,j)$-entry of $A$.
\end{fact}

Hereafter, we assume that $A_{\phi}=\left( \begin{array}{cc}
     a   &  b  \\
      c   &  d \\
  \end{array} \right) $ with $ad\geq 0$ and $|a|\geq |d|$. 
Note that $\alpha \left(=L(1,0)\right)$ is mapped to $L(a,c)$ by $\phi$. 

\begin{rmk}
For the case of closed torus bundles over circles, the homeomorphism types are classified by the conjugation classes of matrices similarly. 
Though they are not the complements of fibered knots, almost same arguments for Theorem~\ref{realpositive} and Theorem~\ref{one} imply that a closed torus bundle over a circle has bi-orderable fundamental group if and only if the trace of the monodromy matrix is grater than or equal to $2$. 
In the other cases, we can easily find generalized torsions since the fundamental group of the torus is abelian: 
\begin{itemize}
\item When $a+d=1$, we assume that $A_{\phi}=\left( \begin{array}{cc}
     1   &  -1  \\
      1   &  0 \\
  \end{array} \right) \left(=\left( \begin{array}{cc}
     1   &  0  \\
      0   &  -1 \\
  \end{array} \right) \left( \begin{array}{cc}
     1   &  1  \\
      -1   &  0 \\
  \end{array} \right) \left( \begin{array}{cc}
     1   &  0  \\
      0   &  -1 \\
  \end{array} \right) \right) $. Then ${A_{\phi}}^{3}+E=O$, where $E$ and $O$ are the identity matrix and zero-matrix, and thus every element of the fundamental group of the closed torus is a generalized torsion.
\item When $a+d\leq 0$, we have ${A_{\phi}}^{2}+(-a-d)A_{\phi}+E=O$ by the Cayley-Hamilton's equation. 
Thus every element of the fundamental group of the closed torus is a generalized torsion.
\end{itemize}
However, in the case of once punctured torus bundles, we should look more carefully since $\pi_{1}(\Sigma,*)$ is not abelian.
\end{rmk}

\subsection{The case where $trA_{\phi}=1$} \label{subsec5_1}
We assume that $A_{\phi}=\left( \begin{array}{cc}
     1   &  -1  \\
      1   &  0 \\
  \end{array} \right)$. 
Then $\alpha \left(=L(1,0)\right)$ is mapped to $L(-1,0)$, $\alpha$ with the opposite orientation, by $\phi^{3}$. 
We can connect $\alpha$ and $\phi^{3}(\alpha)$ by an arc $\gamma$ as in Figure~\ref{trace1}. 
Note that $W\left(l\left(\alpha, \gamma, \phi^{3}(\alpha)\right)\right)$ is the empty word. 
Fix an oriented based path $I_{\alpha}\in \pi_{1}(\Sigma,*)$ such that $F(I_{\alpha})=\alpha$. 
Then $W\left(\phi^{3}(I_{\alpha})\right)=\phi^{3}(\alpha)$ also holds. 
By Lemma~\ref{connection}, there exists an element $g\in \pi_{1}(\Sigma,*)$ such that $W\left(F\left(I_{\alpha}g\phi^{3}(I_{\alpha})g^{-1}\right)\right)$ is the empty word in $C(G')$, and this implies that $w\left(I_{\alpha}g\phi^{3}(I_{\alpha})g^{-1}\right)$ is the empty word in $G$. 
Thus $I_{\alpha}g\phi^{3}(I_{\alpha})g^{-1}=1_{\pi_{1}(\Sigma,*)}$ holds, and we see that $I_{\alpha}$ is a generalized torsion.

\begin{figure}[htbp]
 \begin{center}
  \includegraphics[width=100mm]{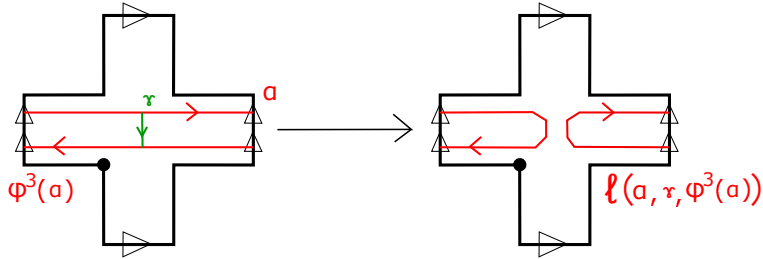}
 \end{center}
 \caption{$\gamma$ and $l\left(\alpha, \gamma, \phi^{3}(\alpha)\right)$.}
 \label{trace1}
\end{figure}

\subsection{The case where $trA_{\phi}\leq 0$}
In this case, $\alpha \left(=L(1,0)\right)$ is mapped to $L(a,c)$ by $\phi$ and to $L(d,-c)$ by $\phi^{-1}$. Note that $a\leq d\leq0$ since $ad\geq0$ and $|a|\geq|d|$. 
If $a=0$, then $\phi^{2}$ reverses the orientation of $\alpha$. Thus we get a generalized torsion as in Subsection~\ref{subsec5_1}. In the following, we assume that $a<0$. 
$W\left( L(a,c)\right)$ consists of $|a|$-times $X^{-1}$ and $|c|$-times $Y^{sign(c)}$,
 $W(\alpha)$ is $X$, and $W\left( L(d,-c)\right)$ consists of $|d|$-times $X^{-1}$ and $|c|$-times $Y^{-sign(c)}$, where $sign(c)$ is $+1$ if $c\geq0$ and $-1$ otherwise. 
Along $L(a,c)$, take $|a|$ points $*_{1},\dots,*_{|a|}$ just after intersection points with $I_{x}$. 
Take $|a|$-copies $\alpha_{1},\dots, \alpha_{|a|}$ of $\alpha\left(=L(1,0)\right)$ and arcs $\gamma_{1},\dots,\gamma_{|a|}$ which are disjoint from $I_x$ and $I_y$ so that $\gamma_{i}$ connects $*_i$ and $\gamma_{i}$ for each $i=1,\dots ,|a|$. See Figure~\ref{connect}. 
Fix an oriented based path $I_{\alpha}\in \pi_{1}(\Sigma,*)$ such that $F(I_{\alpha})=\alpha$. 
By using Lemma~\ref{connection} repeatedly, we see that there exist elements $g_{1},\dots,g_{|a|}\in \pi_{1}(\Sigma,*)$ such that $W\left( F\left( \phi(I_{\alpha})\cdot(g_{1}I_{\alpha}{g_{1}}^{-1}))\cdot \cdots \cdot (g_{|a|}I_{\alpha}{g_{|a|}}^{-1}) \right)\right)=Y^{c}$.

\begin{figure}[htbp]
 \begin{center}
  \includegraphics[width=100mm]{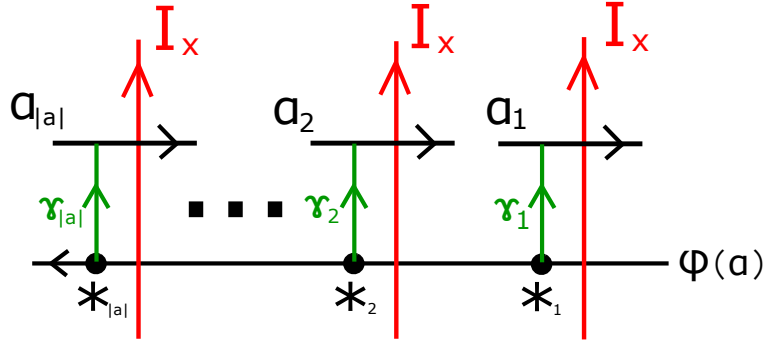}
 \end{center}
 \caption{$*_{1},\dots,*_{|a|}$ and $\alpha_{1},\dots,\alpha_{|a|}$ and $\gamma_{1},\dots,\gamma_{|a|}$.}
 \label{connect}
\end{figure}

If $d=0$, $W\left(\phi^{-1}(\alpha)\right)=Y^{-c}$. 
Then by using Lemma~\ref{connection}, choosing connecting arc which is disjoint from $I_x$ and $I_y$, we see that there exists an element $g\in \pi_{1}(\Sigma,*)$ such that \\$W\left( F\left( \phi(I_{\alpha})\cdot(g_{1}I_{\alpha}{g_{1}}^{-1}))\cdot \cdots \cdot (g_{|a|}I_{\alpha}{g_{|a|}}^{-1}) (g\phi^{-1}(I_{\alpha})g^{-1})\right)\right)$ is the empty word in $C(G')$,
 and this implies that $w\left(\left( \phi(I_{\alpha})\cdot(g_{1}I_{\alpha}{g_{1}}^{-1}))\cdot \cdots \cdot (g_{|a|}I_{\alpha}{g_{|a|}}^{-1}) (g\phi^{-1}(I_{\alpha})g^{-1})\right)\right)$ is the empty word in $G$. 
Thus $\left( \phi(I_{\alpha})\cdot(g_{1}I_{\alpha}{g_{1}}^{-1}))\cdot \cdots \cdot (g_{|a|}I_{\alpha}{g_{|a|}}^{-1}) (g\phi^{-1}(I_{\alpha})g^{-1})\right)=1_{\pi_{1}(\Sigma,*)}$ holds, and we see that $I_{\alpha}$ is a generalized torsion.\\
If $d<0$, like $\phi(\alpha)$, there exist elements $g'_{1},\dots,g'_{|d|}\in \pi_{1}(\Sigma,*)$ such that \\$W\left( F\left( \phi^{-1}(I_{\alpha})\cdot(g'_{1}I_{\alpha}{g'_{1}}^{-1}))\cdot \cdots \cdot (g'_{|d|}I_{\alpha}{g'_{|d|}}^{-1}) \right)\right)=Y^{-c}$. 
Then by Lemma~\ref{connection}, choosing connecting arc which is disjoint from $I_x$ and $I_y$, we see that there exists an element $\tilde{g}\in \pi_{1}(\Sigma,*)$ such that \\
$W\left( F\left( \phi(I_{\alpha})\cdot(g_{1}I_{\alpha}{g_{1}}^{-1})\cdot \cdots \cdot (g_{|a|}I_{\alpha}{g_{|a|}}^{-1})\cdot \tilde{g} \left(\phi^{-1}(I_{\alpha})\cdot(g'_{1}I_{\alpha}{g'_{1}}^{-1})\cdot \cdots \cdot (g'_{|d|}I_{\alpha}{g'_{|d|}}^{-1})\right) \tilde{g}^{-1}\right)\right)$ is the empty word in $C(G')$,
 and this implies that \\
$w\left( \phi(I_{\alpha})\cdot(g_{1}I_{\alpha}{g_{1}}^{-1})\cdot \cdots \cdot (g_{|a|}I_{\alpha}{g_{|a|}}^{-1})\cdot \tilde{g} \left(\phi^{-1}(I_{\alpha})\cdot(g'_{1}I_{\alpha}{g'_{1}}^{-1})\cdot \cdots \cdot (g'_{|d|}I_{\alpha}{g'_{|d|}}^{-1})\right) \tilde{g}^{-1}\right)$ is the empty word in $G$. \\
Thus $\phi(I_{\alpha})\cdot(g_{1}I_{\alpha}{g_{1}}^{-1})\cdot \cdots \cdot (g_{|a|}I_{\alpha}{g_{|a|}}^{-1})\cdot \tilde{g} \left(\phi^{-1}(I_{\alpha})\cdot(g'_{1}I_{\alpha}{g'_{1}}^{-1})\cdot \cdots \cdot (g'_{|d|}I_{\alpha}{g'_{|d|}}^{-1})\right) \tilde{g}^{-1}=1_{\pi_{1}(\Sigma,*)}$ holds, and we see that $I_{\alpha}$ is a generalized torsion.

\begin{rmk} (This remark is suggested by professor Motegi)\\
In general, the tunnel number of every once punctured torus bundle is less than or equal to $2$, where the tunnel number is the minimal number of properly embedded arcs such that the complement is the handlebody. 
We can choose arcs $I_x$ and $I_y$ in one fiber for the complement being a handlebody, for example. 
In \cite{baker}, tunnel number one once punctured torus bundles are completely determined. 
They said that the tunnel number of a once punctured torus bundle is one if and only if there is a simple free loop on a fiber such that this loop and the image of this loop under the monodromy intersect once. 
The monodromy matrix of such a once punctured torus bundle is $\left( \begin{array}{cc}
     m   &  1  \\
      -1   &  0 \\
  \end{array} \right)$ for some $m\in \mathbb{Z}$ under appropriate basis. 
In our result, we can find a tunnel number two once punctured torus bundle whose fundamental group has a generalized torsion: 
For example, consider a once punctured torus bundle whose monodromy matrix is $A=\left( \begin{array}{cc}
     4n-1   &  -2n  \\
     2   &  -1 \\
  \end{array} \right)$ for some $n\leq0$. This has a generalized torsion since the trace is $4n-2\leq 1$. 
Suppose for a contradiction that the tunnel number of this is one. 
Then $A$ is conjugate in $GL_{2}(\mathbb{Z})$ to $B=\left( \begin{array}{cc}
     4n-2   &  1  \\
      -1   &  0 \\
  \end{array} \right)$ or $B^{-1}$. 
Then modulo $2$, every element conjugate to $A$ in $GL_{2}(\mathbb{Z})$ is congruent to $\left( \begin{array}{cc}
     1   &  0  \\
      0  & 1\\
  \end{array} \right)$. 
However, $B$ and $B^{-1}$ are congruent to $\left( \begin{array}{cc}
     0   &  1  \\
      1  & 0\\
  \end{array} \right)$ modulo $2$, this leads a contradiction. 

Moreover, if we suppose $n\leq -1$ in addition, then the monodromy is pseudo-Anosov (not reducible and not periodic). 
Then by \cite{otal} \cite{thurston}, corresponding once punctured torus bundle admits a complete hyperbolic structure of finite volume.

\end{rmk}

\vspace{0.5cm}

\ GRADUATE SCHOOL OF MATHEMATICAL SCIENCES, THE UNIVERSITY OF TOKYO, 3-8-1 KOMABA, MEGURO--KU, TOKYO, 153-8914, JAPAN\\
\ \ E-mail address: \texttt{sekinonozomu@g.ecc.u-tokyo.ac.jp}
\end{document}